\documentclass[12pt,a4paper]{article}
\usepackage{fullpage,amsfonts,amsmath,amssymb,graphicx}
\usepackage[all]{xy}
\usepackage{multirow}
\usepackage{diagbox}
\usepackage{arydshln}
\usepackage{bbold}
\usepackage{fix-cm} 
\usepackage{calrsfs}
\usepackage{calligra}
\DeclareFontFamily{T1}{calligra}{}
\DeclareFontShape{T1}{calligra}{m}{n}{<-> s * [1.80] callig15}{}
\DeclareMathAlphabet{\mathcalligra}{T1}{calligra}{m}{n}
\DeclareFontFamily{OT1}{pzc}{}
\DeclareFontShape{OT1}{pzc}{m}{it}{<-> s * [1.30] pzcmi7t}{}
\DeclareMathAlphabet{\mathpzc}{OT1}{pzc}{m}{it}
\DeclareMathAlphabet{\pazocal}{OMS}{zplm}{m}{n}

\usepackage[english,frenchb]{babel}
\usepackage[utf8x]{inputenc}
\usepackage{authblk}
\usepackage{enumerate}
\usepackage{amsthm}
\newtheorem{theorem}{Theorem}[section]

\newtheorem{proposition}[theorem]{Proposition}
\newtheorem{prop}[theorem]{Proposition}

\newtheorem*{theorem*}{Theorem}
\newtheorem*{lemma*}{Lemma}

\theoremstyle{definition}

  \usepackage{color}

\usepackage{tikz}
\usetikzlibrary{calc}

\tikzset{point/.style = {fill=gray,circle,inner sep=2pt}}

\newcommand\T{\rule{0pt}{2.6ex}} 
\newcommand\B{\rule[-1.2ex]{0pt}{0pt}} 
\newcommand\TopStrut[1]{\rule{0pt}{#1ex}} 

\newcommand\CC{{\mathbb C}}
\newcommand\FF{{\mathbb F}}

\newcommand\PP{{\mathbb P}}
\newcommand\QQ{{\mathbb Q}}
\newcommand\RR{{\mathbb R}}

\newcommand\ZZ{{\mathbb Z}}

\newcommand\JJJ{{\pazocal J}}

\newcommand{\Aut}{\operatorname{Aut}\nolimits}

\newcommand{\Cl}{\operatorname{Cl}\nolimits}

\newcommand\equi{{\ \Longleftrightarrow\ }}
\newcommand{\Fix}{\operatorname{Fix}\nolimits}

\newcommand{\id}{\operatorname{id}\nolimits}

\newcommand{\Stab}{\operatorname{Stab}\nolimits}

\renewcommand\mod{{\,\mathrm{mod}\,}}

\renewcommand\tilde[1]{\widetilde{#1}}
\renewcommand\bar[1]{\overline{#1}}

\newcommand\lra{{\longrightarrow}}

\newlength{\rrrr}

\newcommand{\intoo}[1]{\:
\xymatrix@1{\ar@{^(->}[r]^{#1}&}\:}

\newcommand{\ootni}[1]{\:
\xymatrix@1{&\ar@{_(->}[l]_(.3){#1}}\:}

\title{Action of the automorphism group on the Jacobian of Klein's quartic curve}

\author[$\dagger$]{Dimitri Markushevich} 
\author[$*$]{Anne Moreau}

\affil[$\dagger$]{Univ. Lille, CNRS, UMR 8524 -- Laboratoire Paul Painlev\'e, F-59000 Lille, France; 
e-mail: dimitri.markouchevitch@univ-lille.fr\newline\  }

\affil[$*$]{Facult\'e des Sciences d'Orsay, Universit\'e Paris-Saclay, 91405 Orsay, France; e-mail: 
anne.moreau@universite-paris-saclay.fr}

\batchmode

\begin{document}

\maketitle

\selectlanguage{english}

\section*{Introduction}

Klein's simple group $H_{168}$ of order 168 can be defined by  $H_{168}\simeq \mathbf{PSL}(2,7)\simeq \mathbf{GL}(3,2)$, where $\mathbf{GL}(n,q)$, resp. $\mathbf{PSL}(n,q)$ stands for the linear (resp. projective special linear) group of automorphisms of the $\FF_q$-vector space $\FF_q^n$, where $\FF_q$ is the finite field with $q$ elements. Klein introduced this group in 1879 \cite{Klein} and described its irreducible 3-dimensional complex representation by automorphisms of the plane quartic curve $C\subset\PP^2_\CC$ with equation $x^3y+y^3z+z^3x=0$, called Klein's quartic curve. See, for example \cite{Eightfold} and \cite{Ba-Itz} for a modern exposition, some applications and interesting ramifications.

Klein's simple group also appears in the context of groups generated by complex reflections. Consider it as a complex linear group acting on the 3-dimensional complex vector space $V\simeq \CC^3$, 
whose projectivization is the projective plane containing the Klein quartic: $\PP(V)=\PP^2_\CC$. This representation embeds $H_{168}$ into $\mathbf{SL}(3,\CC)$. If we extend this copy of $H_{168}$ by adding $-\id_{\CC^3}$, we will obtain a subgroup of $\mathbf{GL}(3,\CC)$ of order 336, which we will denote $G_{336}$. This extension of $H_{168}$ is not just split, it is simply a direct product: $G_{336}=\{\pm\id\}\times H_{168}$. In spite of the apparent triviality of this step, it brings in a new very important property: $G_{336}$ is one of the finite complex reflection groups classified by Shephard--Todd  \cite{Sh-To}; see also \cite{Co}
 for a simplified approach to the classification. 

On the other hand, the action of $G_{336}$ on $\CC^3$ is of arithmetic nature, as it preserves a rank-6 lattice in $\CC^3$.  One can easily see the existence of such a lattice $\Lambda$. Indeed, as $H_{168}$ acts on Klein's curve $C$, it also acts on its Jacobian $\JJJ=\JJJ(C)$, a 3-dimensional abelian variety. So we can represent $\JJJ(C)$ as the complex torus $\CC^3/\Lambda$, where $\Lambda$ is the period lattice of $C$, and then the action of $H_{168}$ lifts to a linear action on $\CC^3$ leaving invariant $\Lambda$. The fact that the action on $\CC^3$ is the same as that on $V$ can be verified by using the canonical identification $\JJJ= H^0(C,\Omega_C^1)^*/\Lambda$, and the action on $H^0(C,\Omega_C^1)$  can be deduced from the representation of the 1-forms on $C$ via the Poincaré residue. In this way we recover a lattice $\Lambda$, invariant under $G_{336}$, as the period lattice of $C$ in $H^0(C,\Omega_C^1)^*$.

As $G_{336}$ leaves invariant the lattice $\Lambda$, one can construct the extension $\widetilde{G}_{336}$ of $G_{336}$  by adding the translations by vectors from $\Lambda$:
\begin{equation}\label{G-tilde}
0\lra \Lambda \lra \widetilde{G}_{336} \lra G_{336}\lra 0.
\end{equation}
The thus obtained group $\widetilde{G}_{336}$ of affine transformations of $\CC^3$ is
a complex crystallographic group generated by reflections, or a CCR group for short.
Moreover, $G_{336}$ is the complete group of linear transformations leaving invariant $\Lambda$. Indeed, by Torelli theorem, the order of the automorphism group of $\JJJ(C)$ is twice the order of the automorphism group of $C$, and the latter is 168, which is the maximal order of the automorphism group of a curve of genus $g=3$ by the Hurwitz inequality $|\Aut(C)|\leq 84(g-1)$. 


The main object of interest of the present study is the quotient variety $X=\JJJ/G_{336}$, which can also be viewed as the quotient $\CC^3/\widetilde{G}_{336}$ by the CCR group. 
This quotient can be thought of as the projective spectrum of the algebra of $G_{336}$-invariant theta functions for $\JJJ$. For finite reflection groups acting on $\CC^n$, we have the Chevalley--Shephard--Todd Theorem, which states that the algebra of polynomial invariants of the action is also polynomial, that is freely generated by $n$ basic generators. 
It is a natural conjecture that the analogue of the Chevalley--Shephard--Todd Theorem also holds for irreducible affine CCR groups. The conjecture can be stated in other words by saying that for such a group $\Gamma$, the quotient variety $\CC^n/\Gamma$ is a weighted projective space. This conjecture, taken in full generality, persists for more than 40 years, since Looijenga \cite{Loo} established the result for the CCR groups $\Gamma$ obtained as the extensions of the Weyl group of a real irreducible root system in $\RR^n$ by a complexification of its root lattice. Such complexified real crystallographic reflection groups depend on one complex parameter $\tau$. Several papers generalized and improved this result in several ways, and at present it is known to be true for all CCR groups of Coxeter type (\cite{Be-Sch1}-\cite{Be-Sch3}, \cite{Kac-P}, \cite{W}, \cite{FMW}).

The conjecture was also claimed to be proven in dimension two, see \cite{Schw1}, \cite{TY}, \cite{KTY}, but the proofs were based on an incomplete classification of rank-2 CCR groups. For example, as we know from \cite{Deraux}, \cite{KRR}, the weighted projective plane $\PP(1,3,8)$ is a CCR quotient, but it is missing in the above references; see also \cite{Po}, \cite[\S5]{GM}.
In dimension $>2$, not a single result of this type is known for any one of the genuinely complex crystallographic reflection groups, i.~e. those which are not of Coxeter type. A classification of such groups can be found in \cite{Po}.  By contrast with the CCR groups of Coxeter type, genuinely complex CCR groups are all rigid: there is no continuous parameter $\tau$. According to Popov's classification, there exists a unique complex crystallographic reflection group with point group
$G_{336}$: it is listed as $[K_{24}]$ in Table 2 in loc. cit. (24 being the number of $G_{336}$ in the classification table of \cite{Sh-To}). From Popov's table, one also reads off the generators of the invariant lattice $\Lambda$ and the extension cocycle, which can only be  zero in this case. Thus an extension of $G_{336}$ by $\Lambda$ is always split, so that $\tilde G_{336}=\Lambda\rtimes G_{336}$ is a unique such extension, and the $G_{336}$-invariant lattice $\Lambda$ is unique modulo equivalence. We will use a slightly different, more symmetric representation of $\Lambda$ from \cite{Mazur}.

Our results on the singularities of $X$ make it plausible that $X$ is the weighted projective space $\PP(1,2,4,7)$. We look into the combinatorics of the action of $G_{336}$ and list the stabilizers and the orbits in $\JJJ$. As follows from Theorem \ref{SingX}, $X$ and $\PP(1,2,4,7)$ have the same singularities.

{\em Acknowledgements.} D.~M. was partially supported by the PRCI SMAGP (ANR-20-CE40-0026-01) and the Labex CEMPI  (ANR-11-LABX-0007-01). A.M. is partially supported by ANR Project GeoLie Grant number ANR-15-CE40-0012. The authors thank T.~Dedieu et X.~Roulleau  for discussions.

\section{Klein's group $H_{168}$, its double $G_{336}$ and the invariant lattice $\Lambda$}
\label{prelim-on-H168}

We introduce the group $G=G_{336}$ directly in its embedding in $\mathbf U(3)$ as the group generated by reflections in the roots of the complex root system, usually denoted $J_3(4)$, but we will fix the notation ${\pazocal R}$ for it. 
We describe it following \cite[pp. 235-236]{Mazur}. The root system ${\pazocal R}$ is the set of vectors of $\CC^3$, obtained from $(2,0,0)$, $(0,\alpha,\alpha)$ and $(1,1,\bar\alpha)$, where $\alpha=\frac{1+i\sqrt 7}{2}$, by sign changes and permutations of coordinates. The root lattice $\Lambda=Q({\pazocal R})$ generated by ${\pazocal R}$ can be given by
$$
\Lambda=\{(z_1,z_2,z_3)\in{\pazocal O}^3\ | \ z_1\equiv z_2 \equiv z_3\ \mathrm{mod}\ \alpha,\ 
z_1+z_2+z_3\equiv 0 \ \mathrm{mod}\ \bar\alpha\},
$$
where $\pazocal O=\ZZ+\ZZ\alpha=\ZZ[\alpha]$ is the ring of integers of the quadratic field $K=\QQ(\alpha)$.
The group $G$ is the subgroup of $\mathbf U(3)$ leaving invariant $\Lambda$.
The translations by $\Lambda$ extend $G$ to an affine crystallographic reflection group $\tilde G$.

The standard Hermitian scalar product of $\CC^3$ is not primitive when restricted to $\Lambda$, so 
we will endow $\CC^3$ with the Hermitian scalar product which is half the standard one:
$$
\forall x,\ y\in\CC^3,\ \ (x,y):=\frac12\sum_{i=1}^3\bar x_iy_i.
$$
With these definitions, ${\pazocal R}$ contains 42 roots $e$, all of them being of square 2: $(e,e)=2$. They are divided in 21 pairs of opposite roots $\pm e$. Choosing one representative from each pair in an arbitrary way, we obtain the subset $\pazocal R_0$ of 21 roots which will be called positive roots. The group $G$ is generated by the 21 reflections in the positive roots $e\in \pazocal R_0$,
$$
r_e:\CC^3\to\CC^3,\ \ x\mapsto x-(e,x)e, 
$$
and Klein's simple group is the unimodular part of $G$:
$$
H_{168}=\{h\in G\ | \ \det (h)=1\}.
$$
It can be thought of as the group generated by the 21 antireflections $\rho_e:=-r_e$, or by products $r_er_{e'}$ of pairs of reflections ($e,e'\in \pazocal R_0$). These generating sets are redundant; to generate $G$, it suffices to use three reflections. We choose the three ``basic'' roots as $e_1=(0,\alpha,\alpha)$, $e_2=(0,0,2)$ and $e_3=(1,1,\bar\alpha)$ in such a way that the corresponding generators of $G$ are the same as chosen in \cite[(10.1)]{Sh-To}:
$$
r_1=r_{e_1}={\scriptsize \begin{pmatrix}
        1&0&0\\
        0&0&1\\
        0&1&0\end{pmatrix}},\ 
r_2=r_{e_2}={\scriptsize \begin{pmatrix}
        1&0&0\\
        0&1&0\\
        0&0&{-1}\end{pmatrix}}, \ 
r_3=r_{e_3}=\tfrac12 {\scriptsize \begin{pmatrix}
        1&{-1}&{-\alpha}\\
        {-1}&1&{-\alpha}\\
        -\bar\alpha&-\bar\alpha&0\end{pmatrix}}.
$$
These generators satisfy the following relations:
$$
r_1^2=r_2^2=r_3^2=(r_1r_2)^4=(r_2r_3)^4=(r_3r_1)^3=(r_1r_2r_1r_3)^3=1.
$$
By loc. cit., this is a presentation of $G$ by generators and relations.

Obviously, $\rho_i=-r_i$ ($i=1,2,3$) generate ${H_{168}}$. As a minimal set of generators of ${H_{168}}$ one can choose 
$$r_3r_1=\tfrac12{\scriptsize \begin{pmatrix}
        1&{-\alpha}&{-1}\\
        {-1}&{-\alpha}&1\\
        -\bar\alpha&0&-\bar\alpha\end{pmatrix}}\  \mbox{and}\ \ r_1r_2={\scriptsize \begin{pmatrix}
        1&0&0\\
        0&0&{-1}\\
        0&1&0\end{pmatrix}}; \ \ (r_3r_1)^3=(r_1r_2)^4=1.$$

The orders of elements of $G$ are 1, 2, 3, 4, 6, 7, 14.
An example of an element of maximal order in $G$ (an analogue of a Coxeter element) is
\begin{equation}\label{ord14}
r_1r_2r_3= \tfrac12{\scriptsize \begin{pmatrix}
        1&{-1}&{-\alpha}\\
        \bar\alpha&\bar\alpha&0\\
        {-1}&1&{-\alpha}\end{pmatrix}},\ \ 
        (r_1r_2r_3)^7=-1.
\end{equation}

Remark that $\Lambda$ is a free $\pazocal O$-module of rank 3, generated by the basic roots $e_1,e_2,e_3$ introduced above:
$$
\Lambda={\pazocal O} e_1+{\pazocal O} e_2+{\pazocal O} e_3.
$$
This representation of $\Lambda$ implies that the elements of $H$ and $G$ can be given by matrices from $M_3(\pazocal O)$ in the basis $(e_1,e_2,e_3)$. The disadvantage of this representation is that it is not unitary. 
So we stick to the representation of $G$ by unitary matrices in the standard basis of $\CC^3$ from which we started, though the elements of these unitary matrices are half-integers from $\pazocal O$. The columns of each matrix in $G$ are roots from $\pazocal R$ divided by 2, so making a complete list of elements of $G$ amounts to the enumeration of all the triples of mutually orthogonal roots in $\pazocal R$.
Over $\ZZ$, we will fix $$(\epsilon_1,\ldots,\epsilon_6)=(\alpha e_1,\alpha e_2,\alpha e_3,\bar\alpha e_1,\bar\alpha e_2,\bar\alpha e_3)$$ as the ``standard'' $\ZZ$-basis of $\Lambda$.

The famous equation of Klein's quartic $x^3y+y^3z+z^3x=0$ is referred to coordinates in which an order-7 element of ${H_{168}}$ is diagonalized with eigenvalues $\zeta,\ \zeta^4,\ \zeta^2$,
where $\zeta=\exp\frac{2\pi i}7$, but in the coordinates used in our representation it becomes
$$
x^4 + y^4 + z^4 - 3 \bar\alpha (x^2y^2 + x^2z^2 + y^2z^2)=0.
$$

The next table from \cite{CL} provides a list of the 15 conjugacy classes of subgroups of ${H_{168}}$ with their minimal overgroups and maximal subgroups; these data determine a structure of a lattice (partially ordered set) on the set of subgroups of ${H_{168}}$. The notation for groups used in the column ``Structure'' is standard for papers in the theory of finite groups; we explain some of them that are unusual in other fields of mathematics: $n$ is a cyclic group of order $n$; $m^n$ is the direct product of $n$ copies of a cyclic group of order $m$; $N:L$ is a semi-direct product of $N$ and $L$ with $N$ a normal subgroup; $L_n(q)$ is what we denote $PSL(n,q)$, so that $L_2(7)\simeq {H_{168}}$. The repetition of a type of a subgroup means that there are two orbits under conjugation, their lengths are given in the column ``Length''. The last two columns refer to subgroups by their numbers from the first column, the integers between parentheses indicating the number of distinct subgroups of given type that are minimal overgroups or maximal subgroups for the subgroup from the current line.
\begin{center}
\begin{tabular}{|c|l|c|c|l|l|}
\hline
Nr.&Structure&Order&Length&Maximal Subgroups&Minimal Overgroups\\
\hline
1&$L_2(7)$&168&1& 2 (7), 3 (7), 4 (8) & \\
\hline
2&$2^2:S_3$&24&7&5, 7 (3), 9 (4)&1\\
\hline
3&$2^2:S_3$&24&7&6, 7 (3), 9 (4)&1\\
\hline
4&$7:3$&21&8&8, 13 (7)&1\\
\hline
5&$A_4$&12&7&10, 13 (4)&2 \\
\hline
6&$A_4$&12&7&11, 13 (4)&3 \\
\hline
7&$D_8$&8&21&10, 12, 11&2, 3\\
\hline
8&7&7&8&15&4\\
\hline
9&$S_3$&6&28&13, 14 (3)&2, 3\\
\hline
10&$2^2$&4&7&14 (3)&5, 7 (3)\\
\hline
11&$2^2$&4&7&14 (3)&6, 7 (3)\\
\hline
12&4&4&21&14&7\\
\hline
13&3&3&28&15&4 (2), 5, 6, 9\\
\hline
14&2&2&21&15&9 (4), 10, 11, 12\\
\hline
15&1&1&1&&8 (8), 13 (28), 14 (21)\\
\hline
\end{tabular}
\end{center}
\smallskip

We will not list all the subgroups of $G$, but just note that each subgroup $K$ of ${H_{168}}$ has a degree-two extension in $G$, denoted $\pm K$:
$$\pm K=\langle -1, K\rangle=\{\pm k\ | \ k\in K\}\simeq\{\pm 1\}\times K.$$
Of course, $G$ also has other types of subgroups.

For future reference, we provide some explicit examples of subgroups of ${H_{168}}$ from the table: 
\begin{multline}
D_8=\langle s,t\ | \ s^4=t^2=1,\ tst=s^{-1}\rangle=\\ 
\{1,s=h_4,h_4^2,h_4^3,t=\rho_1,\rho_2,\rho_2h_4,h_4\rho_1\},\
h_4=\rho_1\rho_2;
\label{D8}
\end{multline}
\begin{equation}\label{ord7}
7\simeq G_7=\{1,g_7,\ldots,g_7^6\},\ g_7=\rho_1\rho_2\rho_3=-r_1r_2r_3=
{\scriptsize \tfrac12\begin{pmatrix}
{-1}&1&\alpha\\
-\bar\alpha&-\bar\alpha&0\\
1&{-1}&\alpha\end{pmatrix}};
\end{equation}
\begin{equation}\label{ord21}
7:3\simeq G_{21}=\langle g_7,h_3\ | \ g_7^7=h_3^3=1,\ h_3g_7h_3^{-1}=g_7^2\rangle,\  \ h_3=\rho_1\rho_3\rho_1\rho_2;
\end{equation}
\begin{multline}
2^2:S_3\simeq S_4\simeq G_{24}=\left\{ \gamma={\scriptsize \begin{pmatrix}
        \pm1&0&0\\
        0&{\pm1}&0\\
        0&0&{\pm1}\end{pmatrix}, \  
        \begin{pmatrix}
        \pm1&0&0\\
        0&0&{\pm1}\\
        0&\pm1&0\end{pmatrix}, \  
        \begin{pmatrix}
        0&0&{\pm1}\\
        0&{\pm1}&0\\
        {\pm1}&0&0\end{pmatrix}},\right. \\
        \left.\left. {\scriptsize \begin{pmatrix}
        0&{\pm1}&0\\
        {\pm1}&0&0\\
        0&0&{\pm1}\end{pmatrix}, \ 
         \begin{pmatrix}
        0&0&{\pm1}\\
        {\pm1}&0&0\\
        0&\pm1&0\end{pmatrix}}, \ \mbox{or} \ \
         {\scriptsize \begin{pmatrix}
        0&{\pm1}&0\\
        0&0&{\pm1}\\
        {\pm1}&0&0\end{pmatrix}}
        \ \ \right| \ \ \det \gamma = 1\right\}.\label{monomial}
\end{multline}

In the next table we list the conjugacy classes of $H$:\smallskip

\begin{center}
\begin{tabular}{|c|c|c|c|c|c|c|}
\hline
\ \ \ \ ord\,$(\gamma)$\ \ \ \ & 1 & 2 & 3 & 4 & 7 & 7 \\
\hline
\ \ \ \ $|\mathrm{Cl}_H(\gamma)|$\ \ \ \  &1&21&56&24&24&24\ \\
\hline
\ \ \ \  \ $\gamma$\ \  \ \ \  \ &1&$\rho_1$&$h_3$&$h_4$&$g_7$&$g_7^{-1}$\ \\
\hline
\end{tabular}
\end{center}
\smallskip
The representatives $h_3,h_4,g_7$ are defined in \eqref{D8}-\eqref{ord21}. The conjugacy classes of $G$ are deduced from these in an obvious way: to every conjugacy class $\Cl_H(\gamma)$ in $H$ correspond two conjugacy classes in $G$ of the same length: $\Cl_G(\gamma)=\Cl_H(\gamma)$ and $\Cl_G(-\gamma)=-\Cl_H(\gamma)$.

\section{Fixed loci of elements of $G_{336}$ acting on $\JJJ=\CC^3/\Lambda$}

We divide the elements of $G$ in two classes, elliptic and parabolic; the parabolic ones are defined as those having 1  among their eigenvalues, and all the remaining elements are called elliptic. The elliptic elements are $-1$, the 42 elements of order 4 with determinant $-1$, the 56 elements of order 6, and those of order 7 and 14. There are also 42 elements of order 4 with determinant 1, but they are parabolic. For both orders 7 and 14, there are two conjugacy classes of length 24, but what we need for enumerating the fixed points is the number of cyclic subgroups generated by them, and there are fewer classes of elliptically generated cyclic subgroups. 

\begin{prop}\label{elliptic}
$G$ has the following cyclic subgroups generated by elliptic elements:
\begin{enumerate}[i)]
\item One group of order $2$, $C_2=\{\pm 1\}$, with $64$ fixed points in $\JJJ$ that are images of the half-periods of $\Lambda$:
$$
\{\xi_0,\ldots,\xi_{63}\}=\left\{ \sum_{i=1}^6x_i\epsilon_i,\ x_i\in\left\{0,\tfrac12\right\}\right\}.
$$
\item One conjugacy class of $21$ cyclic subgroups of order $4$, $C_4^{(1)},\ldots,C_4^{(21)}$, having each $16$ fixed points in $\JJJ$. Choosing $C_4^{(1)}=\langle h'_4\rangle$, $h'_4=-r_1r_2:(z_1,z_2,z_3)\mapsto (-z_1,z_3,-z_2)$, we find the representatives of the $16$ fixed points of $h'_4$ in the form
$$
\{\beta_0,\ldots,\beta_{15}\}=\iota_0(1,0,0)+\iota_1(\alpha,0,0)+\iota_2(\tfrac{\alpha}2,\tfrac{\alpha}2,-\tfrac{\alpha}2)
+\iota_3(\tfrac{\bar\alpha}2,1,0), \ \ \iota_k\in\left\{0,1\right\}.
$$
\item One conjugacy class of $28$ cyclic subgroups of order $6$, $C_6^{(1)},\ldots,C_6^{(28)}$, having each $4$ fixed points in $\JJJ$. Choosing $C_6^{(1)}=\langle c\rangle$ with $c=(z_1,z_2,z_3)\mapsto (-z_3,-z_1,-z_2)$
we identify the representatives of the $4$ fixed points in $\Pi$ as:
$$
\omega_{ij}=\frac{i}2(\bar\alpha,\bar\alpha,\bar\alpha)+j(1,1,1), \ (i,j)\in\{0,1\}^2,
$$
so that $\omega_{00}=0$ and the remaining $3$ points $\omega_{ij}$ belong to the set of $64$ fixed points of $C_2$ from item i).
\item One conjugacy class of $8$ subgroups $C_7^{(1)},\ldots,C_7^{(8)}$ of order $7$, having each $7$ fixed points on $\JJJ$. Choosing $C_7^{(1)}=\langle g_7\rangle$, where $g_7$ is defined in \eqref{ord21}, we find the following representatives of the $7$ fixed points of $C_7^{(1)}$:
$$
\eta_0=0,\ \eta_i=\frac17\Big(-i\epsilon_1-i\epsilon_2+i\epsilon_3+i\epsilon_4+i\epsilon_5-i\epsilon_6\Big),
 \ i=1, \dots, 6.
$$
\item One conjugacy class of $8$ cyclic subgroups $C_{14}^{(1)},\ldots,C_{14}^{(8)}$ of order $14$, having each a unique fixed point, the zero of $\JJJ$.
\end{enumerate}
\end{prop}

\begin{proof}
Let $\gamma$ be an elliptic element and $z\in \CC^3$ a fixed point of $\gamma$ modulo $\Lambda$. This means that
$\gamma z-z\in\Lambda$, or else $z\in (\gamma -\id_{\CC^3})^{-1}(\Lambda)$. Thus the number of fixed points modulo $\Lambda$ is equal to $[(\gamma -\id_{\CC^3})^{-1}(\Lambda):\Lambda]$.
Hence to evaluate the number of fixed points on $\JJJ$, it suffices to calculate the determinant of $\gamma -\id_\Lambda$, where $\gamma$ is viewed as an automorphism of the rank-6 $\ZZ$-module $\Lambda$. When working with $3\times 3$ complex matrices, this determinant becomes $|\det (\gamma -\id_{\CC^3})|^2$. The calculation of $\det (\gamma -\id_{\CC^3})$ for $\gamma = -1,\ h'_4,\ c,\ {g_7},\ -{g_7}$ gives, respectively, the values $-8,\ -4,\ -2,\ i\sqrt{7},\ -1$. This implies the assertion on the numbers of fixed points. The explicit representatives produced in the statement are obtained by a direct calculation. It is quite easy for the orders $<7$, and for order $7$, we wrote down ${g_7}$ by an integer matrix in the $\ZZ$-basis $(\epsilon_i)$ of $\Lambda$ and searched for the fixed points in the unit cube of $\RR^6$. We omit further details.
\end{proof}

The non-elliptic elements different from 1 have fixed loci of positive dimension in $\JJJ$, which are translates of elliptic curves or abelian surfaces. We denote the eigenspace of $\gamma$ corresponding to an eigenvalue $\lambda$ by $V^{(\gamma)}_\lambda$, or simply by $V_\lambda$. We also denote $\Lambda^{(\gamma)}_\lambda$ or just $\Lambda_\lambda$ the intersection
$\Lambda\cap V^{(\gamma)}_\lambda$. When this is a full-rank lattice
in $V^{(\gamma)}_\lambda$, the quotient $\JJJ^{(\gamma)}_\lambda=\JJJ_\lambda:=V^{(\gamma)}_\lambda/\Lambda^{(\gamma)}_\lambda$ is an abelian variety of dimension $\dim V^{(\gamma)}_\lambda$. 

When $\lambda=1$ is among the eigenvalues of $\gamma$, $\JJJ^{(\gamma)}_1$ is the connected component of 0 in the fixed locus $\JJJ^\gamma=\Fix_\JJJ(\gamma)$, but the latter fixed locus can contain several connected components, which are translates of $\JJJ^{(\gamma)}_1$. 
The number of components can be determined as follows. Let $\Lambda^{(\gamma)}_{\mathrm a},V^{(\gamma)}_{\mathrm a}$ be the anti-invariant parts of $\gamma$ in $\Lambda$, respectively $V$, that is the orthogonal complements of $\Lambda^{(\gamma)}_1,V^{(\gamma)}_1$,  and $\JJJ_{\mathrm a}=V^{(\gamma)}_{\mathrm a}/\Lambda^{(\gamma)}_{\mathrm a}$ (the superscript ${(\gamma)}$ can be omitted when there is no risk of confusion). Then
$\JJJ_1$ and $\JJJ_{\mathrm a}$ are complementary in the sense that $\JJJ_1+\JJJ_{\mathrm a}=\JJJ$ and $\JJJ_1\cap \JJJ_{\mathrm a}$ is finite. As the action of $\gamma$ restricted to $\JJJ_{\mathrm a}$ is elliptic, we can determine the number of fixed points $\#\JJJ^\gamma_{\mathrm a}$ for this action as we did before, for example by computing the determinant of $(\gamma-\id_{\CC^3})|_{V_{\mathrm a}}$. Then we have for the group of connected components of $\JJJ^\gamma$:
$$
\JJJ^\gamma/\JJJ_1\simeq \JJJ^\gamma_{\mathrm a}/(\JJJ^\gamma_{\mathrm a}\cap \JJJ_1).
$$
Hence to know the number of components, we have to determine $\#(\JJJ^\gamma_{\mathrm a}\cap \JJJ_1)$, that is, the number of points of $\JJJ^\gamma_{\mathrm a}$ whose representatives in $\CC^3$ are zero modulo $V_1+\Lambda$.

 For reflections the eigenspace $V_1$ is a plane, in which case we call it the mirror. For the remaining non-trivial parabolic elements $\gamma$, $V_1$ is 1-dimensional, and we call it the axis of $\gamma$. For both reflections and antireflections, we have $V_{\mathrm a}=V_{-1}$.

\begin{prop}\label{parabolic}
$G$ has the following cyclic subgroups of order $>1$ generated by parabolic elements:
\begin{enumerate}[i)]
\item One conjugacy class of $21$ subgroups of order $2$ generated by reflections; the fixed locus in $\JJJ$ of each of them is the abelian surface, the image of the mirror of the reflection.
\item One conjugacy class of $21$ subgroups of order $2$ generated by antireflections; the fixed locus in $\JJJ$ of each of them is the union of $4$ translates of the elliptic curve $\JJJ_1$ in $\JJJ$, the image of the axis of the antireflection.
\item One conjugacy class of $28$ subgroups of order $3$; the fixed locus in $\JJJ$ of each of them is the elliptic curve  $\JJJ_1$, the image of the axis of the generator.
\item One conjugacy class of $21$ cyclic subgroups of order $4$; the fixed locus in $\JJJ$ of each of them is the elliptic curve  $\JJJ_1$, the image of the axis of the generator.
\end{enumerate}
\end{prop}

\begin{proof} 

i) As the reflections form one conjugacy class, it suffices to compute the fixed locus just for one of them; choose 
$r_2:(z_1,z_2,z_3)\mapsto(z_1,z_2,-z_3)$.
For $z=(z_1,z_2,z_3)\in\CC^3$, the point $z+\Lambda\in \JJJ=\CC^3/\Lambda$ is fixed under $r_2$
if and only if $r_2(z)-z=-z_3(0,0,2)\in\Lambda$, which is equivalent to $z_3\in{\pazocal O}$. Then there exists
$v\in V_1=\ker (r_2-\id_{\CC^3})$ such that $z=v+z_3(\bar\alpha,1,1)$, hence $z\equiv v\mod\Lambda$ and thus
$z$ represents the point $v+\Lambda$ of the abelian surface $\JJJ_1=V_1/(V_1\cap\Lambda)$, the image of the mirror $V_1$ in $\JJJ$. We see that the restricted action on $\JJJ_{\mathrm a}=\JJJ_{-1}$ is by multiplication by $-1$, so the fixed locus $\JJJ^{r_2}_{\mathrm a}$ consists of 4 points, images of the half-periods of $\Lambda_{-1}$, but all the 4 fixed  points are contained in $\JJJ_1$, so $\JJJ^{r_2}_{\mathrm a}/(\JJJ^{r_2}_{\mathrm a}\cap \JJJ_1)$ is trivial.

ii) Compute the fixed locus of $\rho_2=-r_2$. Here $V_1$ is the $z_3$-axis. For $z=(z_1,z_2,z_3)\in\CC^3$, the point $z+\Lambda\in \JJJ=\CC^3/\Lambda$ is fixed under $\rho_2$
if and only if $\rho_2(z)-z=(-2z_1,-2z_2,0)\in\Lambda$, which is equivalent to 
$$
(z_1,z_2,0)\in\tfrac12\Lambda_{\mathrm a}, \ \mbox{where}\ \Lambda_{\mathrm a}:= \Lambda\cap \{ z_3=0\}={\pazocal O}(2,0,0)+
{\pazocal O}(\alpha,\alpha,0).
$$
The latter condition means that $(z_1,z_2,0)$, modulo $\Lambda_{\mathrm a}$, is one of the 16 linear combinations of the 
vectors
$$
(1,0,0),\ (\bar\alpha,0,0),(\tfrac\alpha2,\tfrac\alpha2,0), (1,1,0)
$$
with coefficients from $\{0,1\}$. As $(1,1,0)\equiv (1,1,\bar\alpha)\mod V_1$ and $(1,1,0)+(1,0,0)+(\bar\alpha,0,0)
\equiv (\bar\alpha,1,1)+(2,0,0)\mod V_1$, we see that only four of the 16 linear combinations are distinct modulo $V_1+\Lambda$, which implies the conclusion.

iii) We will determine the fixed locus of the order-3 element $c^4=-c$, where $c$ is the order-$6$ element from Prop. \ref{elliptic} iii).
For $z\in\CC^3$ the property of being a fixed point of the order-3 element $-c$ modulo $\Lambda$ can be given the following characterization:
$$
(z_3-z_1,z_1-z_2,z_2-z_3)\in\Lambda_{\mathrm a}=\Lambda\cap\{z_1+z_2+z_3=0\}={\pazocal O}(\alpha,0,-\alpha)+{\pazocal O}(0,\alpha,-\alpha).
$$
Looking at the induced action on the abelian surface $\JJJ_{\mathrm a}$, we easily find 9 fixed points, whose representatives modulo $\Lambda_{\mathrm a}$ can be given by
$$
\theta_{ij}=\tfrac i3(-\alpha,-\alpha,2\alpha)+\tfrac j3 (-2,-2,4), \ i,j=0,1,2
$$
The existence of exactly nine fixed points for the induced action on $\JJJ_{\mathrm a}$ can be confirmed by the calculation of the determinant of $(-c-\id)|_{V_{\mathrm a}}$. Now we easily see that the $\theta_{ij}$ are $0$ modulo $V_1+\Lambda$, for example, 
$\theta_{1,0}=\tfrac 13(-\alpha,-\alpha,2\alpha)=-(\alpha,\alpha,0)+\frac23(\alpha,\alpha,\alpha)$, where
$-(\alpha,\alpha,0)\in\Lambda$ and $\frac23(\alpha,\alpha,\alpha)\in V_1$. Hence the images of $\theta_{ij}+V_1$ in $\JJJ$ are one and the same elliptic curve passing through zero.

iv) We will determine the fixed locus of the order-4 parabolic element $h_4=-h'_4$, where $h'_4$ was defined in Prop. \ref{elliptic} ii): $h_4:(z_1,z_2,z_3)\mapsto (z_1,-z_3,z_2)$. Here $V_1$ is the $z_1$-axis. A point $z\in\CC^3$ is fixed under $h_4$ modulo $\Lambda$ if and only if
$$
(0,z_2+z_3,z_3-z_2)\in \Lambda_{\mathrm a}=\Lambda\cap\{z_1=0\}={\pazocal O}(0,2,0)+{\pazocal O}(0,\alpha,\alpha).
$$
There are 4 solutions modulo $\Lambda_{\mathrm a}$: $0$, $(0,1,1)$, $(0,0,\alpha)$, and $(0,1,1+\alpha)$. All of them are in $\Lambda+V_1$, for example, $(0,1,1)=(\bar\alpha,1,1)+(-\bar\alpha,0,0)$ with $(\bar\alpha,1,1)\in\Lambda$,
$(-\bar\alpha,0,0)\in V_1$. Hence the fixed locus of $h_4$ is connected.
\end{proof}

\section{Orbits with elliptic stabilizers}

We want to enumerate all the possible stabilizers $G_u=\Stab_G(u)$ and $H_u=\Stab_{H}(u)$ of points $u\in \JJJ$. In this section we will consider the points $u$ fixed by at least one elliptic element of $G$. Such points and their stabilizers will be called {\em elliptic}. In the case when the stabilizer $G_u$ is non-trivial but contains no elliptic elements, we will call $u$ and its stabilizer {\em parabolic}. The parabolic points will be studied in the next section.

The knowledge of the stabilizer provides the length of the orbit of $u$, which is the index of the stabilizer, and determines the singularities of the quotient varieties at $G\cdot u$ and ${H}\cdot u$, the orbits of $u$ viewed as points of the respective quotients that are the images of $u$. The image of $u$ is a nonsingular point of the quotient if and only if the stabilizer is generated by reflections, otherwise it is a singularity, locally analytically equivalent to the linear quotient $\CC^3/G_u$, resp. $\CC^3/H_u$. 

The points of a Zariski open set of $\JJJ$ have trivial stabilizer in $G$ or ${H}$; we call this Zariski open set the {\em free locus} of $G$, resp. ${H}$.

The non-free locus of $G$ is the union of two-dimensional images of mirrors of reflections, of a number of curves and of a number of isolated points. The union of images of mirrors will be called the {\em discriminant arrangement} in $\JJJ$. By Prop. \ref{parabolic} i), the discriminant arrangement is the union of 21 abelian surfaces passing through zero, which we will also call, by abuse of language, {\em mirrors} or {\em mirror abelian surfaces}.
A generic point of a mirror abelian surface which is the image of the mirror of a reflection $r$ has minimal stabilizer, equal to
$\langle r\rangle$.
 The stabilizer can jump along some curves, called  {\em special curves}. The special curves that belong to the discriminant arrangement are the intersection curves of two or more mirrors. Such curves are called  {\em special discriminant curves}.
 
The points of a special curve with stabilizer bigger than that of the generic point of the curve will be called {\em dissident points} of the special curve. 
The curve components of the non-free locus will be called {\em off-discriminant special curves}, and the points of the zero-dimensional irreducible components of the non-free locus will be called {\em isolated special points}.

We will also distinguish the points $u$ of the non-free locus in $\JJJ$ according to the property whether their stabilizer $G_u$ in $G$ is cyclic or not; we will say that $u$ is a {\em cyclic point} if $G_u$ is a cyclic subgroup of $G$. 
The most special point is $0\in \JJJ$; it is stabilized by the whole of $G$ and is a smooth point of $X=\JJJ/G$, as $G$ is generated by reflections.  

Now we are turning to the locus of elliptic points. It turns out that the only {\em isolated} special points are the elliptic cyclic points fixed by elements of order 7. They are treated in the next Proposition; we determined six of them in Prop. \ref{elliptic} iv), and all of them belong to the orbits of these six. 

We denote by $C_d$ a cyclic group of order $d$, and by $\frac1d(\nu_1,\nu_2,\nu_3)$ the (analytic equivalence class of the) cyclic quotient singularity $\CC^3/C_d$, where
the generator $c_d$ of $C_d$ acts by 
$c_d:(z_1,z_2,z_3)\mapsto (\epsilon^{\nu_1}z_1,\epsilon^{\nu_2}z_2,\epsilon^{\nu_3}z_3)$,\ \  $\epsilon=\exp\big(\frac{2\pi i}d\big)$.

\begin{proposition} \label{T7}
Let $T_7$ denote the set of $48$ non-zero points of $\JJJ$ fixed by elements of order $7$.
\begin{enumerate}[i)]
\item Suppose $\eta\in T_7$ is fixed under the action of an element $\sigma\in {H_{168}}$ of order $7$. Then
$\Stab_{H_{168}}(\eta)=\langle \sigma\rangle$ is of order $7$, so $T_7$ is the union of two ${H_{168}}$-orbits of length $24$.
\item In the notation of i), the normalizer $N_{H_{168}}(\langle \sigma\rangle)\simeq G_{21}$, where $G_{21}$ is the group of order $21$ introduced in \eqref{ord21}, and there exists an element $\tau$ of order $3$ in $N_{H_{168}}(\langle \sigma\rangle)$ such that
$\tau(\eta)=2\eta$, hence $\eta, \ 2\eta, \ 4\eta$ belong to one of the two ${H_{168}}$-orbits in $T_7$, while $3\eta,\ 5\eta,\ 6\eta$ belong to the other. As representatives of the two orbits, one can choose $\eta_1$ and $\eta_3$, where
$\eta_i$ ($i=1,\ldots,6$) were introduces in Prop. \ref{elliptic} iv), and we denote by the same symbol $\eta_i$ the fixed points of ${g_7}$ on $\JJJ$ represented by the vectors $\eta_i\in\CC^3$.
\item The images in $Y=\JJJ/{H_{168}}$ of the $2$ ${H_{168}}$-orbits in $T_7$ are $2$ isolated cyclic quotient singularities of $Y$ of local analytic type $\frac17(1,2,4)$.
\item The action of $-1$ permutes the two ${H_{168}}$-orbits, hence $T_7$ is just one $G$-orbit, whose image in the quotient $X=\JJJ/G$ is an isolated singularity of local analytic type $\frac17(1,2,4)$.
\end{enumerate}
\end{proposition}

\begin{proof}
For any element $\sigma$ of order 7, $\Fix_\JJJ(\sigma)$ is the same as $\Fix_\JJJ(\langle\sigma\rangle)$. As there is only one conjugacy class of subgroups of order 7 in ${H_{168}}$, we can restrict ourselves to one particular element of order 7, say the element ${g_7}$ introduced in \eqref{ord21}. So we assume $\sigma={g_7}$. As we know that ${H_{168}}$ has eight 7-Sylow subgroups, the order of the normalizer of $\langle {g_7}\rangle$ is 21. The formula \eqref{ord21} presents a subgroup of order 21 normalizing $\langle {g_7}\rangle$, which is $G_{21}$, hence $N_{H_{168}}(\langle {g_7}\rangle)= G_{21}$. The order-3 element ${h_3}$ from  \eqref{ord21} normalizes $\langle {g_7}\rangle$, hence leaves invariant $\Fix_\JJJ(\langle {g_7}\rangle)$. By a direct calculation we check that ${h_3}$ doubles each fixed point of $\langle {g_7}\rangle$. Indeed,
expressing $\eta_1$ in coordinates of $\CC^3$, we obtain 
$$
\eta_1=\begin{pmatrix}
      \frac{i\sqrt7}7\\
      \frac{7+i\sqrt7}{14}\\
      1-\frac{2i\sqrt7}7\end{pmatrix},\ 
{h_3}=\frac12 \begin{pmatrix}
      1&{-\alpha}&1\\
      -\bar\alpha&0&-\bar\alpha\\
      {-1}&{-\alpha}&{-1}\end{pmatrix}, \ 
{h_3}(\eta_1)-2\eta_1=
      \begin{pmatrix}
      1-\alpha\\
      -1-\alpha\\
      -3+\alpha\end{pmatrix}\in\Lambda.
$$
As $\eta_i=i\eta_1$ for all $i=1,\ldots,6$, we deduce that the $\langle {h_3}\rangle$-orbit of $\eta_1$ consists of the three points $\eta_1,\eta_2,\eta_4$. This implies i) and ii), and the remaining assertions easily follow.
\end{proof}

We will now compute the stabilizers of the remaining points from fixed loci of elliptic elements. The next proposition uses the notation of Proposition \ref{elliptic}. 

\begin{prop}\label{dissidents}
\begin{enumerate}[i)]
\item The locus $T_6$ of non-zero fixed points of elements of order $6$ in $G$ is the union of orbits of the $3$ fixed points $\omega_{ij}\ ((i,j)\neq(0,0))$ of the order-$6$ element $c=(z_1,z_2,z_3)\mapsto (-z_3,-z_1,-z_2)$. We have: $\Stab_{H_{168}}(\omega_{10})\simeq
\Stab_{H_{168}}(\omega_{11})\simeq S_4$, \ $\Stab_{H_{168}}(\omega_{01})=\Stab_{H_{168}}(\omega_{10})\cap
\Stab_{H_{168}}(\omega_{11})\simeq S_3$. The three points are contained in $\JJJ^{\langle-1\rangle}$ and
are special on the off-discriminant special curve $\JJJ_1^{(-c)}$, the image of
the axis $z_1=z_2=z_3$ of $-c:(z_1,z_2,z_3)\mapsto (z_3,z_1,z_2)$. Moreover, $\omega_{10}$ and $\omega_{11}$ are
quadruple points of the configuration of special curves, as they  each belong to and are dissident on three special discriminant curves which are fixed by the order-$4$ elements in their stabilizers. Say, for $\omega_{10}$, the stabilizer is nothing else but the monomial subgroup \eqref{monomial}, the three axes of its $6$ order-$4$ elements are just the coordinate axes of $\CC^3$, and the three extra special curves passing through $\omega_{10}$ are the images of the coordinate axes.
The stabilizers in $G$ are twice bigger, $\Stab_G(\omega_{ij})=\pm\Stab_{H_{168}}(\omega_{ij}):=
\{\pm1\}\times\Stab_{H_{168}}(\omega_{ij})$, and they are generated by reflections, so that the images of $\omega_{ij}$ in
$X$ are smooth points.
\item The locus $T'_4$ of non-zero fixed points of elements of order $4$ with determinant $-1$ is the union of the orbits of the $15$ fixed points $\beta_i\ (i\neq 0)$ of the order-$4$ element $h'_4$. In the notation of $\beta_i$ we will understand $i$ as a binary multiindex $\iota_0\iota_1\iota_2\iota_3$ varying from $0000$ to $1111$. The next table lists the stabilizers of $\beta_i$ (except for $\beta_{0000}=0$), up to isomorphism, and the singularities at the images of the corresponding points $\beta_i$ in $X$. We mark by the plus sign the $\beta_i$ that are fixed by $-1$; the numbers between brackets in the last line indicate the number of images  in $X$ of the points $\beta_i$ from the current column; $D_8$, $D'_8$ denote dihedral groups of order 8, the first of which is a subgroup of ${H_{168}}$, the second is not; similarly for the pair $S_4,S'_4$.\smallskip

\begin{center}
\begin{tabular}{|c|c|c|c|c|c|}
\hline
\ $\Stab_{H_{168}}(\beta_i)$\ \  & $S_4$ &$ A_4$  & $D_8$ & $C_2\times C_2 $& $C_2$ \T\\
\hline
\ $\Stab_G(\beta_i)$\ \  & $\pm S_4$ &$ S'_4$\T  & $\pm D_8$ & $D'_8$ & $C_4$ \\
\hline
$\beta_i$  
 &\shortstack{$\beta_{0100}(+)$ \\ $\beta_{1100}(+)$}
      &\T\TopStrut{10}\B\shortstack{$\beta_{0001}$\\ $\beta_{0010}$ \\ $\beta_{0110}$\\ $\beta_{1101}$}&$\beta_{1000}(+)$&
      \shortstack{$\beta_{1010}$\\ $\beta_{0101}$\\ $\beta_{1001}$\\ $\beta_{1110}$}&\shortstack{$\beta_{0011}$\\
      $\beta_{1011}$\\ $\beta_{0111}$\\ $\beta_{1111}$} \\
\hline
 {\rm Image in $X$}& {\rm smooth [2]}& {\rm smooth [2]}& {\rm smooth [1]}&{\rm smooth [2]}&${\frac14(1,2,3)\ [1]}$\B\T\\
\hline
\end{tabular}
\end{center}
\smallskip
All the $G$-stabilizers except for $C_4$ are generated by reflections and the corresponding points $\beta_i$ are mapped to smooth points of $X=\JJJ/G$. 
The image in $X$ of the points $\beta_i$ with stabilizer $C_4$ is a non-isolated cyclic quotient singularity of analytic type $\CC^3/C_4$, where $C_4$ acts with weights $1,2,3$.
\item The locus $T_2$ of $63$ points fixed by the action of $-1$ on $\JJJ\setminus\{0\}$ decomposes into the following $G$-orbits:
\begin{itemize}
\item the two orbits of the points $\beta_{0100}, \beta_{1100}$ from ii) (or of $\omega_{10}$, $\omega_{11}$ from i)) with $G$-stabilizers $\pm S_4$, of length $7$ each;
\item the orbit of the point $\beta_{1000}$ with $G$-stabilizer $\pm D_8$ of length $21$;
\item the orbit of the point $\omega_{01}$ from i) with $G$-stabiliser $\pm S_3$ of length $28$.
\end{itemize}
These stabilizers are generated by reflections, so the image of $T_2$ in $X$ consists of $4$ smooth points.
\end{enumerate}
\end{prop}

\begin{proof}
i) The first assertion follows from the fact that the elements of order 6 form one orbit under conjugation by ${H_{168}}$ (and by $G$).
All the remaining assertions but the last one are proved by a routine verification, which we performed using the computer algebra system Macaulay2 \cite{M2}. For the last assertion, remark that the groups $S_3,S_4$ are generated by their elements of order 2, and all the elements of order 2 in ${H_{168}}$ are antireflections. Hence the stabilizers of $\omega_{ij}$ in ${H_{168}}$ are generated by antireflections. Passing to the stabilizers of  $\omega_{ij}$ in $G$, we extend the stabilizers in ${H_{168}}$ by adding $-1$, and this obviously provides groups generated by reflections.

ii)  As in i), the proof is obtained by a computer-assisted enumeration of the elements of the stabilizers, followed by the inspection of the elements of order 2. 

iii) All the points of $T_2$ belong to orbits already enumerated in i), ii), so iii) is an obvious consequence of i), ii).
\end{proof}

\section{Parabolic orbits and singularities of $\JJJ/G_{336}$}

In the previous section, we enumerated all the {\em elliptic} special points in $\JJJ$. All of them, except for those belonging to the orbit in the last column of the table in Proposition \ref{dissidents} ii),  turn out to be non-cyclic, that is have non-cyclic stabilizer in $G$. Now we will enumerate the parabolic points.

An obvious way to obtain a curve whose generic point is non-cyclic is to take the intersection of two mirror abelian surfaces fixed by reflections. Recall what happens in the case when the two reflections, say $r,r'$, commute: they generate a subgroup $(\ZZ/2\ZZ)^2$, their product $\rho=rr'$ is an anti-reflection, and there is a unique cyclic subgroup of order 4 in $H$ containing $\rho$. This follows from the description of the lattice of subgroups of $H$ in Section 1. Thus the curve which is the intersection of the mirrors of two commuting reflections $r,r'$ can be also characterized as the image $\JJJ_1^{(\rho)}$ in $\JJJ$ of the axis of the antireflection $\rho=rr'$, and the full fixed locus $\JJJ^{\rho}$ of $\rho$ is the union of four translates of the elliptic curve $\JJJ_1^{(\rho)}$ (Proposition \ref{parabolic}, ii)).

We will start by enumerating the parabolic points $u$ with cyclic $H_u$.

\begin{proposition}\label{spe-para}
Let $u\in\JJJ$ be a parabolic point, and assume that $H_u$ is cyclic. Then one of the following three cases is realized:
\begin{enumerate}[(a)]
\item $H_u=\langle\rho\rangle$ is of order $2$. In this case $\rho$ is an anti-reflection and its fixed locus $\JJJ^{\rho}$ is the disjoint union of $4$ translates $\kappa_i+\JJJ_1^{(\rho)}$ ($i=0,1,2,3$) of the elliptic curve $\JJJ_1^{(\rho)}$. The points $\kappa_i$ can be choosen in such a way that the following is true:
$\kappa_0=0$, $\kappa_1$, $\kappa_2$, $\kappa_3=\kappa_1+\kappa_2$ are points of order $2$, and $u$
belongs to one of three curves $\kappa_i+\JJJ_1^{(\rho)}$, $i=1,2,3$.
For generic $u_i\in \kappa_i+\JJJ_1^{(\rho)}$,  $i=1,2$, the $H$-stabilizer $H_{u_i}=\langle\rho\rangle$ is of order $2$, while the $G$-stabilizer $G_{u_i}\simeq (\ZZ/2\ZZ)^2$ is generated by two reflections $r_i,r'_i$ such that $\rho=r_ir'_i$.
For generic $u_3\in \kappa_3+\JJJ_1^{(\rho)}$, the $H$- and $G$-stabilizers coincide: $G_{u_3}=H_{u_3}=\langle\rho\rangle=G_{u_1}\cap G_{u_2}$.
For all the three curves $\kappa_i+\JJJ_1^{(\rho)}$, $i=1,2,3$, the subgroup of $H$ leaving invariant each of them is isomorphic to $D_8$.
\item $u\in \JJJ^{c_3}$ for some element $c_3\in H$ of order $3$, $H_u=\langle c_3\rangle$, and $G_u$ is of type $S'_3$ (a subgroup, isomorphic to $S_3$ and not contained in $H$). The subgroup of $H$ (resp. $G$) leaving invariant $ \JJJ^{c_3}$ is of type $S_3$ (resp. $\pm S_3$).
\item $u\in\JJJ^{c_4}$ for some element $c_4\in H$ of order $4$, $H_u=\langle c_4\rangle$, and $G_u$ is of type $D'_8$. The subgroup of $H$ (resp. $G$) leaving invariant $\JJJ^{c_4}$ is $D_8$ (resp. $\pm D_8$), where we denote, as before, by $D_8$ (resp. $D'_8$) a dihedral subgroup of order $8$ embedded in $H$ (resp. in $G$ in such a way, that the image contains four reflections).

\end{enumerate}
In the cases (b), (c), $G_u$ is generated by reflections and the image of $u$ in $X$ is nonsingular.  In the case (a), the subgroups $G_{u_1}, G_{u_2}$ are generated by reflections and $G_{u_3}$ is not, where $u_i$ denotes a generic point of the curve $\kappa_i+\JJJ_1^{(\rho)}$, so the images of $u_1,u_2$ in $X$ are nonsingular and the image of $u_3$ is a non-isolated singularity of type $\frac12(1,1,0)$.
\end{proposition}

\begin{proof}
The cyclic subgroups of $H$ are all conjugate to those generated by $\rho_1,h_3,h_4$ or $g_7$. Only $\rho_1,h_3,h_4$ are parabolic. We have $|G_u|=2|H_u|$ or $G_u=H_u$.  In the case $|H_u|=2$, we have $H_u=\langle\rho\rangle$ for an element $\rho$ of order 2; all the 21 elements of order 2 in $H$ are anti-reflections conjugate to $\rho_1$, so we may assume $\rho=\rho_1$. It is impossible that $u\in \JJJ_1^{(\rho_1)}$, because every element of order 2 in $H$ is the square of an element of order 4 fixing the same axis, and hence $u$ would then be fixed by a subgroup of order 4 in $H$  at least.
Hence $u$ belongs to 
$\JJJ^{\rho_1}\setminus  \JJJ_1^{(\rho_1)}$, which is the union of the three translates of $\JJJ_1^{(\rho_1)}$ according to Proposition \ref{parabolic} ii): 
$$
[(1,0,0)]+\JJJ_1^{(\rho_1)},\ [(\tfrac\alpha2,\tfrac\alpha2,0)]+\JJJ_1^{(\rho_1)},\ [(1+\tfrac\alpha2,\tfrac\alpha2,0)]+\JJJ_1^{(\rho_1)}.
$$
We can set $\kappa_1=[(1,0,0)]$, $\kappa_2=[(\tfrac\alpha2,\tfrac\alpha2,0)]$; the assertions about the stabilizers are verified by a direct calculation. This provides the case (a).

If $|H_u|=3$, then $H_u=\langle c_3\rangle$ for some element $c_3$ of order 3. Each element of order 3 is a product of two reflections, so $G_u\supset\langle r,c_3\rangle\simeq S_3$, where $r$ is one of those reflections. Let $K=\langle -r,c_3\rangle$. Obviously, $K\simeq \langle r,c_3\rangle\simeq S_3$. From the table of Section 1 describing the lattice of subgroups of $H$, we see that each $3$ is a subgroup of index 2 in a unique $S_3$, its normalizer. The subgroups $S_3$ form one orbit in $H$, so we may choose $K=\langle -r_1,c_3\rangle$,
where $r_1$ is one of our basic reflections and $c_3=c^4=-c$ is the same order-$3$ element as the one used in the proof of Proposition \ref{parabolic} iii). We saw there that the fixed locus $\JJJ^{c_3}$ is the elliptic curve obtained as the image of the diagonal locus of points $(x,x,x)\in\CC^3$ in $\CC^3/\Lambda$. Now $z=(z_1,z_2,z_3)+\Lambda$ is fixed under $r_1:(z_1,z_2,z_3)\mapsto (z_1,z_3,z_2)$ if and only if $r_1(z)-z\in\Lambda$, or $(0,z_3-z_2,z_2-z_3)\in\Lambda$. Obviously, this condition is automatically satisfied for any $z$ of the form $(x,x,x)$, which implies that $G_u\supset\langle r,c_3\rangle\simeq S_3$. This provides the case (b).

By a similar argument, assuming $|H_u|=4$, we reduce the proof to the case when $H_u= \langle h_4\rangle$, where $h_4$ is the element of order $4$ from the proof of Proposition \ref{parabolic} iv). The axis of $h_4$ is the first coordinate axis of $\CC^3$, and one easily verifies that $G_u=D'_8$ for generic point $u$ of the form $(z_1,0,0)+\Lambda$. For non-generic points of this form the stabilizer may be bigger,  but then $H_u$ is non-cyclic, and as we will see in the next proposition, this implies that $u$ is non-parabolic, so all such cases have been treated in the previous section.
\end{proof}

Now we consider the case when $H_u$ is non-cyclic.

\begin{proposition}\label{non-cyclic}
Let $u\in\JJJ$, $u\neq 0$ and assume $H_u$ non-cyclic. Then one of the following cases is realized.
\begin{enumerate}[(a)]
\addtocounter{enumi}{3}
\item $H_u$ contains $S_3$. Then $u\in T_6$, where $T_6$ is the locus of nonzero points fixed by elements of order $6$. This locus, described in Proposition \ref{dissidents} i), is the union of orbits of the three points  $\omega_{01},\omega_{10},\omega_{11}$  with $G$-stabilizers $\pm S_4$ or $\pm S_3$.
\item $H_u$ contains $(\ZZ/2\ZZ)^2$. Then $u$ belongs to the orbit of one of the $16$ fixed points of the elliptic order-$4$ element $h'_4$ from Proposition \ref{elliptic} ii), and the possible $G$-stabilizers of $u$ are $D'_8$,
$\pm D_8$, $S'_4$ and $\pm S_4$. 
\end{enumerate}
In particular, none of these points $u$ is parabolic. Their $G$-stabilizers are generated by reflections, so their images in $X$ are smooth points.
\end{proposition}

\begin{proof}
As $H_u$ is non-cyclic, it contains at least two distinct cyclic subgroups generated by elements from the orbits of $\rho_1,h_3,h_4,g_7$ or $g_7^{-1}$. We can disregard the elements of order 7, because for a nonzero fixed point of such an element, its stabilizer is of order 7 and hence is cyclic. So, we have to consider only the cases when $H_u$ contains two cyclic subgroups of orders 2, 3 or 4.

The first case we will consider is when $H_u$ contains subgroups of orders 2 and 3. From the table in Section 1 describing the lattice of subgroups of $H$ we see that then $H_u$ is one of the subgroups $S_3$, $A_4$, $S_4$. 

Assume that $H_u\supset S_3$. As the subgroups $S_3$ form one orbit, we can choose $S_3=\langle -r_1,c_3\rangle$ as in the proof of the previous proposition. As before, $\JJJ^{c_3}$ is the elliptic curve obtained as the image of the diagonal of $\CC^3$, that is the locus of points of the form $(x,x,x)$ modulo $\Lambda$, and $z=(z_1,z_2,z_3)+\Lambda$ is fixed under $-r_1:(z_1,z_2,z_3)\mapsto (-z_1,-z_3,-z_2)$ if and only if $r_1(z)+z\in\Lambda$, or $(2z_1,z_2+z_3,z_3+z_2)\in\Lambda$. For a point $z$ of the form $(x,x,x)$ the latter condition is equivalent to $x(2,2,2)\in\Lambda$, which gives four points stabilized by $S_3$:
$$
\JJJ^{S_3}=\left\{\iota_1\tfrac{(\bar\alpha,\bar\alpha,\bar\alpha)}2+\iota_2(1,1,1)\right\}_{\iota_1,\iota_2=0,1}
\ \ (\mod \Lambda).
$$
We now see that the three of these points different from 0 belong to the locus $T_6$ from Proposition \ref{dissidents} i), which ends the proof for the case when $H_u\supset S_3$.

We will not consider separately the cases $H_u\supset A_4$ or $S_4$, because in these cases $H_u$ contains a subgroup $\simeq (\ZZ/2\ZZ)^2$. So we will just consider one case when $H_u\supset (\ZZ/2\ZZ)^2$.

There are two orbits of subgroups $2^2$ in $H$, and respectively two orbits of their normalizers $S_4$. 
For each ``positive'' root $e\in\pazocal R_0$, there are two pairs $(e',e'')$, $(f',f'')$ of orthogonal roots in $\pazocal R_0$, such that
$e'\perp e''$, $f'\perp f''$, roots from different pairs being non-orthogonal. Say, if $e=(2,0,0)$, then, for an appropriate choice of $\pazocal R_0$, the two orthogonal pairs are
$(e',e'')=((0,2,0),(0,0,2))$ and $(f',f'')=((0,\alpha,\alpha),(0,\alpha,-\alpha))$. We can choose for representatives of the two orbits of $2^2$ in $H$ the subgroups $H_{2^2}=\langle \rho_{e'},\rho_{e''}\rangle$ and
$H'_{2^2}=\langle \rho_{f'},\rho_{f''}\rangle$, and  $H_{2^2}\cap H'_{2^2}=\langle \rho_e\rangle$. We have:
$$
(z_1,z_2,z_3)+\Lambda\in\JJJ^{H_{2^2}}\equi
(2z_1,2z_2,0)\equiv (0,2z_2,2z_3)\equiv 0 \mod \Lambda
$$
$$
\equi z=(z_1,z_2,z_3)\in \bar\Lambda=\ZZ\tfrac{(\alpha,\alpha,\alpha)}2+\pazocal O^3.
$$
As $[\bar\Lambda:\Lambda]=16$, $\#\JJJ^{H_{2^2}}=16$. Similarly one verifies that $\#\JJJ^{H'_{2^2}}=16$. Moreover, by inspecting the $G$-stabilizers of the 16 fixed points, we observe that each of them contains at least one elliptic element of order 4.
All such elements are conjugate to $h'_4$, thus the possible stabilizers $G_u$ in (d) are those appearing in Proposition \ref{elliptic} ii), except for $D_u\simeq C_4$, for which $H_u\simeq C_2$ is too small.

It remains to consider the case when $H_u$ contains two cyclic subgroups, one of which has order 4. Denoting by $c_4$ a generator of the latter subgroup of order 4, we see that $H_u$ contains two distinct cyclic subgroups, one of which is of order 2, generated by $c_4^2$, and this brings us to one of the cases treated above.
\end{proof}

Now we are ready to enumerate the singularities of the quotient variety $X$. 
We say that a variety (as always in this paper, over $\CC$) is strongly simply connected if its smooth locus is connected and simply connected.

\begin{theorem}\label{SingX}
The quotient $X=\JJJ/G$ is a normal strongly simply connected variety whose singular locus is the union of two irreducible components, $\PP^1=\ell$ and an isolated point $p$. Denoting $\pi:\JJJ\to X$ the natural map, we have $p=\pi(T_7)$ and $\ell=\pi(\kappa_3+\JJJ_1^{(\rho)})$, where $T_7$ is the orbit of fixed points of elements of order $7$, described in Prop. \ref{T7}, $\rho$ is an anti-reflection and $\kappa_3+\JJJ_1^{(\rho)}$ is the elliptic curve in the fixed locus of $\rho$ defined in Prop. \ref{spe-para} (a).\par
The singularity at $p$ is of analytic type $\frac17(1,2,4)$. At all but one points of $\ell$, the singularity of $X$ is of type $\frac12(1,0,1)$, that is $ \CC\times A_1$, the Cartesian product of  $\CC$ with a surface du Val singularity of type $A_1$.
The unique point $q$ of $\ell$ where the type of singularity changes is the image of the orbit of one of the points $\beta_{\iota_0\iota_1\iota_2\iota_3}$ from the last column of the table in Prop. \ref{dissidents} ii), say $\beta_{0011}$.  The type of singularity at $q$ is $\frac14(1,2,3)$.
\end{theorem}

\begin{proof}
The strong simply-connectedness follows from \cite[Theorem 3.2.1]{TY}; see also \cite{Schw} or  \cite[Prop. 0.1]{Be-Sch3}. In fact, for the quotients of $\CC^n$ by complex crystallographic groups, the property of the group to be generated by affine complex reflections is {\em equivalent} to the strong simply-connectedness of the quotient. 

Singularities of $X$ may only occur in the image of the points of $\JJJ$ whose $G$-stabilizers are not generated by reflections. We made a complete inventory of possible $G$-stabilizers. The orbits of points whose $G$-stabilizers are not generated by reflections are those mentioned in the statement of the theorem.
\end{proof}

We note that the weighted projective space $\PP(1,2,4,7)$ is also strongly simply connected and has the same singularities as $X$, which provides some evidence towards the conjecture stated in the introduction.

\end{document}